\documentclass{amsart}
\usepackage{amssymb}
\usepackage{amsfonts}

\newtheorem{theorem}{Theorem}
\theoremstyle{plain}

\newtheorem{corollary}{Corollary}

\newtheorem{definition}{Definition}

\newtheorem{lemma}{Lemma}

\newtheorem{remark}{Remark}

\numberwithin{equation}{section}
\input{tcilatex}

\begin{document}
\title{Some Integral Inequalities via Caputo and Liouville Fractional
Integral Operators for $m-$convex functions}
\author{M.Emin \"{o}zdemir }
\email{eminozdemir@uludag.edu.tr}
\curraddr{Bursa Ulu\i da\u{g} \"{U}niversity, Department of Mathematics\\
and Science Education , Bursa-T\"{u}rkiye}
\thanks{This paper is in final form and no version of it will be submitted
for publication elsewhere.}

\begin{abstract}
This short study consists of two parts, firstly we obtain some inequalities
on Caputo Fractional derivatives using the elementary inequalities. Secondly
we establish several new inequalities including Caputo fractional
derivatives for $m-$Convex functions. In general, in this work we obtain
upper bounds for the left sides of Lemma 1$\left[ 10\right] $ and lemma 2$%
\left[ 20\right] $ .
\end{abstract}

\maketitle

\section{Introduction}

In mathematical analysis, we know roughly that the classical concept of
derivative can be expressed in a single way as the limit of the slopes of
secant lines for $\vartriangle x\rightarrow 0.$ When it comes to Fractional
Derivatives (FC ),

Caputo left-sided derivative%
\begin{equation*}
^{C}D_{^{a^{+}}}^{\alpha }\left[ f\right] (x)=\frac{1}{\Gamma \left(
n-\alpha \right) }\dint_{a}^{x}\left( x-\xi \right) ^{n-\alpha -1}\frac{%
d^{n}\left( f\left( \xi \right) \right) }{d\xi ^{n}}dt,\;\;x>a
\end{equation*}

Caputo right-sided derivative

\begin{equation*}
^{C}D_{^{b-}}^{\alpha }\left[ f\right] (x)=\frac{\left( -1\right) ^{n}}{%
\Gamma \left( n-\alpha \right) }\dint_{x}^{b}\left( \xi -x\right) ^{n-\alpha
-1}\frac{d^{n}\left( f\left( \xi \right) \right) }{d\xi ^{n}}dt,\;\;x<b
\end{equation*}%
it has to do with the concept of tangent.\ As can be seen,\ In Caputo, she
first calculated the derivative of the integer order and then the integral
of the non-integer order.

Liouville left-sided derivative%
\begin{equation*}
I_{^{a^{+}}}^{\alpha }\left[ f\right] (x)=\frac{1}{\Gamma \left( n-\alpha
\right) }\frac{d^{n}}{dx^{n}}\dint_{a}^{x}\left( x-\xi \right) ^{n-\alpha
-1}f\left( \xi \right) d\xi ,\;\;x>a
\end{equation*}

Liouville right-sided derivative%
\begin{equation*}
I_{^{b-}}^{\alpha }\left[ f\right] (x)=\frac{\left( -1\right) ^{n}}{\Gamma
\left( n-\alpha \right) }\frac{d^{n}}{dx^{n}}\dint_{x}^{b}\left( \xi
-x\right) ^{n-\alpha -1}f\left( \xi \right) d\xi ,\;\;x<b
\end{equation*}

In Liouville, the opposite of the process in Caputo is valid. That is, first
the integral of the non-integer order is calculated and then the derivative
of the integer order is calculated.

Since the Caputo Fractional derivative is more restritive than the Liouville
one, both derivatives are defined by means of the each other.%
\begin{equation}
^{C}D_{^{a^{+}}}^{\alpha }\left[ f\right] (x):I_{^{a^{+}}}^{n-\alpha }\left[
f^{\left( n\right) }\right] (x)  \label{d6}
\end{equation}

\begin{equation*}
^{C}D_{^{b-}}^{\alpha }\left[ f\right] (x):\left( -1\right)
^{n}I_{^{b^{-}}}^{n-\alpha }\left[ f^{\left( n\right) }\right] (x)
\end{equation*}

\qquad \qquad \qquad \qquad \qquad \qquad \qquad \qquad \qquad \qquad \qquad
\qquad \qquad \qquad \qquad \qquad \qquad \qquad \qquad \qquad \qquad \qquad
\qquad \qquad \qquad \qquad \qquad \qquad \qquad \qquad \qquad \qquad \qquad
\qquad \qquad \qquad \qquad \qquad \qquad \qquad \qquad \qquad \qquad \qquad

Specially, $\alpha =0\ $then the left and right Caputo derivatives are equal
to each other.

Today, the concept of FC dates back to Leibniz. Leibniz discussed the
concept of FC with her contemporaries in 1695. \ Euler noticed in 1738 what
a problem non-integer order derivatives (FC) pose. By 1822, Fourier gave the
first definition of non-positive integers by using integral notation to
define the derivative. Abel in 1826 and Liouville in 1832 gave versions of
the non-integer order derivative.

Let us present the necessary definitions and preliminary information that we
will use in this study.

\begin{definition}
$\left[ 9\right] \ $The function $f:[0,b]\rightarrow \mathbb{R}$ is said to
be m-convex, where $m\in \lbrack 0,1]$, if for all $x,y$ $\in $ $[0,b]$ and $%
t\in $ $[0,1]$, we have:

\begin{equation*}
\ \ \ \ \ \ \ \ f(tx+m(1-t)y)\leq tf(x)+m(1-t)f(y)\ \ \ \ \ \ \ \ \ \ \ \ \
\ \ \ \ \ \ \ \ \ \ \ \ \ \ \ \ \ \ 
\end{equation*}
\end{definition}

For many papers connected with $m$-convex and $(\alpha ,m)-$convex functions
see $([1-8])$ and the references therein.

We will need the modified forms of the $m$-convex function:

$m-$convexity of $f\ \ :$%
\begin{equation*}
\ f\left( tx+(1-t)y\right) =f(tx+m(1-t)\frac{y}{m})\leq tf(x)+m(1-t)f(\frac{y%
}{m})\ 
\end{equation*}

$m$-convexity of\ $\left\vert f^{\left( n+1\right) }\right\vert \ :$%
\begin{equation*}
\left\vert \ f^{\left( n+1\right) }\left( tx+(1-t)y\right) \right\vert
=\left\vert f^{\left( n+1\right) }(tx+m(1-t)\frac{y}{m})\right\vert \leq
t\left\vert f^{\left( n+1\right) }\left( x\right) \right\vert
+m(1-t)\left\vert f^{\left( n+1\right) }(\frac{y}{m})\right\vert
\end{equation*}

$m$-convexity of\ $\left\vert f^{\left( n+1\right) }\right\vert ^{q}\ .:$%
\begin{equation*}
\left\vert \ f^{\left( n+1\right) }\left( tx+(1-t)y\right) \right\vert
^{q}=\left\vert f^{\left( n+1\right) }(tx+m(1-t)\frac{y}{m})\right\vert
^{q}\leq t\left\vert f^{\left( n+1\right) }\left( x\right) \right\vert
^{q}+m(1-t)\left\vert f^{\left( n+1\right) }(\frac{y}{m})\right\vert ^{q}
\end{equation*}

\begin{definition}
$[11]\ $Let $\alpha \geq 0$ and $\alpha \notin $ $\{1,2,3,...\}$, $n=[\alpha
]+1$, $f\in AC^{n}[a,b]$, the space of functions having $n-$ th derivatives
absolutely continuous. The left-sided and right-sided Caputo fractional
derivatives of order $\alpha $ are defined as follows:
\end{definition}

\begin{equation}
(^{C}D_{^{a^{+}}}^{\alpha }f)(x)=\frac{1}{\Gamma \left( n-\alpha \right) }%
\dint_{a}^{x}\frac{f^{\left( n\right) }\left( t\right) }{\left( x-t\right)
^{\alpha -n+1}}dt,\;\;x>a\;\ \ \ \ \ \ \ \   \label{d3}
\end{equation}

and \ \ \ \ \ \ \ \ \ 
\begin{equation*}
(^{C}D_{^{b^{-}}}^{\alpha }f)(x)=\frac{\left( -1\right) ^{n}}{\Gamma \left(
n-\alpha \right) }\dint_{x}^{b}\frac{f^{\left( n\right) }\left( t\right) }{%
\left( t-x\right) ^{\alpha -n+1}}dt,\;\;x<b
\end{equation*}%
\ \ \ \ \ \ \ \ \ \ \ \ \ \ \ \ \ \ \ \ \ \ \ \ \ \ \ \ \ \ \ \ \ \ \ \ \ \
\ \ \ \ \ \ \ \ \ \ \ \ \ \ \ \ \ \ \ \ \ \ \ \ \ \ \ \ \ \ \ \ \ \ \ \ \ \
\ \ \ \ \ \ \ \ \ \ \ \ \ \ \ \ \ \ \ \ \ \ \ \ \ \ \ \ \ \ \ \ \ \ \ \ \ \
\ \ \ \ \ \ \ \ \ \ \ \ \ \ \ \ \ \ \ \ \ \ \ \ \ \ \ \ \ \ \ \ \ \ \ \ \ \
\ \ \ \ \ \ \ \ \ \ \ \ \ \ \ \ \ \ \ \ \ \ \ \ \ \ \ \ \ \ \ \ \ \ \ \ \ \
\ \ \ \ \ \ \ \ \ \ \ \ \ \ \ \ \ \ \ \ \ \ \ \ \ \ \ \ \ \ \ \ \ \ \ \ \ \
\ \ \ \ \ \ \ \ \ \ \ \ \ \ \ \ \ \ \ \ \ \ \ \ \ \ \ \ \ \ \ \ \ \ \ \ \ \
\ \ \ \ \ \ \ \ \ \ \ \ \ \ \ \ \ \ \ \ \ \ \ \ \ \ \ \ \ \ \ \ \ \ \ \ \ \
\ \ \ \ \ \ \ \ \ \ \ \ \ \ \ \ \ \ \ \ \ \ \ \ \ \ \ \ \ \ \ \ \ \ \ \ \ \
\ \ \ \ \ \ \ \ \ \ \ \ \ \ \ \ \ \ \ \ \ \ \ \ \ \ \ \ \ \ \ \ \ \ \ \ \ \
\ \ \ \ \ \ \ \ \ \ \ \ \ \ \ \ \ \ \ \ \ \ \ \ \ \ \ \ \ \ \ \ \ \ \ \ \ \
\ \ \ \ \ \ \ \ \ \ \ \ \ \ \ \ \ \ \ \ \ \ \ \ \ \ \ \ \ \ \ \ \ \ \ \ \ \
\ \ \ \ \ \ \ \ \ \ 

\ \ \ \ \ \ \ \ \ \ \ \ \ \ \ \ \ \qquad \qquad \qquad \qquad \qquad \qquad
\qquad \qquad \qquad \qquad \qquad \qquad \qquad \qquad \qquad \qquad \qquad
\qquad \qquad \qquad \qquad \qquad \qquad \qquad \qquad \qquad \qquad \qquad
\qquad \qquad \qquad \qquad \qquad \qquad \qquad \qquad \qquad \qquad

\bigskip

\bigskip If $n=1\ $and $\alpha =0\;,\ $we have $%
(^{C}D_{^{a^{+}}}^{0}f)(x)=(^{C}D_{^{b^{-}}}^{0}f)(x)=f\left( x\right) $%
\bigskip\ . For many papers connected fractional operators see $\left( \left[
12-25\right] \right) $

We will also use the well-known H\"{o}lder inequality in the literature :
let be $p>1\;$and $p^{-1}+q^{-1}=1,\ $If$\;f\ \;$and $g\ $reel functions on $%
\left[ a,b\right] \;$such\ that $\left\vert f\ \right\vert ^{p}\;and\ \
\left\vert f\ \right\vert ^{q}\;$are integrable \ on$\ \left[ a,b\right] .$

Then 
\begin{equation*}
\dint_{a}^{b}\left\vert f\left( x\right) g\left( x\right) \right\vert dx\leq
\left( \int\nolimits_{a}^{b}\left\vert f\left( x\right) \right\vert
^{p}dx\right) ^{\frac{1}{p}}\left( \int\nolimits_{a}^{b}\left\vert g\left(
x\right) \right\vert ^{q}dx\right) ^{\frac{1}{q}}.
\end{equation*}

\bigskip\ \ \ \ In $\left[ 10\right] ,$Farid et al. established the
following identity for Caputo fractional operators. \ 

\begin{lemma}
\ In $\left[ 10\right] \ $Let $f:[a,b]\rightarrow \mathbb{R},$ be a
differentiable mapping on (a, b) with $a<b$. If $\ f^{(n+1)}\in L[a,b]$,
then the following equality for fractional integrals holds:%
\begin{eqnarray*}
&&\frac{f^{\left( n\right) }\left( a\right) +f^{\left( n\right) }\left(
b\right) }{2}-\frac{\Gamma \left( n-\alpha +1\right) }{2\left( b-a\right)
^{n-\alpha }}\left[ (^{C}D_{^{a^{+}}}^{\alpha }f)(b)+\left( -1\right)
^{n}(^{C}D_{^{b^{-}}}^{\alpha }f)(a)\right] \ \ \ \ \ \ \ \ \ \ \ \ \ \ \ \
\ \ \ \ \ \ \ \ \ \ \ \ \ \ \ \ \ \ \ \ \ \ \ \ \ \ \ \ \ \ \ \ \ \ \ \ \ \
\ \ \ \ \ \ \ \ \ \ \ \ \ \ \ \ \ \ \ \ \ \ \ \ \ \ \ \ \ \  \\
&=&\frac{b-a}{2}\dint_{0}^{1}\left[ \left( 1-t\right) ^{n-\alpha
}-t^{n-\alpha }\right] f^{\left( n+1\right) }\left( tx+\left( 1-t\right)
y\right) dt\ \ \ \ \ \ \ \ \ \ \ \ \ \ \ \ \ \ \ \ \ \ \ \ \ \ \ \ \ \ \ \ \
\ \ \ \ \ \ \ \ \ \ \ \ \ \ \ \ \ \ \ \ \ \ \ \ \ \ \ \ \ \ \ \ \ \ \ \ \ \
\ \ \ \ \ \ \ \ \ \ \ \ \ \ \ \ \ \ \ \ \ \ \ \ \ \ 
\end{eqnarray*}
\end{lemma}

\begin{lemma}
$\ $In $\left[ 20\right] \ $Let $f:I\subset \mathbb{R}\rightarrow \mathbb{R}$
be a differentiable mapping on I ,where $a,b\in I\ $ with $t\in \left[ 0,1%
\right] .\ $If $\ f^{(n+1)}\in L[a,b]$, Then for all $a\leq x<y\leq b\;$and $%
\alpha >0\ $we have 
\begin{equation*}
\frac{1}{y-x}f^{\left( n\right) }\left( y\right) -\frac{\left( -1\right)
^{n}\Gamma \left( n-\alpha +1\right) }{\left( y-x\right) ^{n-\alpha +1}}%
\left( C_{D_{y^{-}}^{\alpha }}f\right) \left( x\right) =\dint_{0}^{1}\left(
1-t\right) ^{n-\alpha }f^{\left( n+1\right) }\left( tx+\left( 1-t\right)
y\right) dt.
\end{equation*}
\end{lemma}

This work is a continuation of my work in $\left[ 20\right] .\;$\"{O}zdemir
et al. constructed an identity for left sided Caputo derivatives in Lemma 2.
In this study, we constructed differently a few inequalities for both right
and left sided Caputo derivatives. The aim of this paper is to establish new
upper bounds. To do this, we used some classical inequalities.

\section{\protect\bigskip The Results}

\begin{theorem}
Let $f:I\subset \mathbb{R}\rightarrow \mathbb{R},I\subset \lbrack 0,\infty )$
,be a differentiable function on I such that $f\in AC^{n}L[a,b]$ where $%
a,b\in I$ with $0<a<t<x\leq b$. \ If $\alpha >0$ and $\alpha \notin $ $%
\{1,2,3,...\}$, $n=[\alpha ]+1,\;f^{\left( n\right) }>0.$Then%
\begin{equation}
\dint_{a}^{b}f^{\left( n\right) }\left( t\right) dt\leq  \label{d5}
\end{equation}%
\begin{equation*}
\frac{\Gamma \left( n-\alpha \right) \left[ (^{C}D_{^{a^{+}}}^{\alpha
}f)(x)+\left( -1\right) ^{n}(^{C}D_{^{b^{-}}}^{\alpha }f)(x)\right] +\Gamma
\left( \alpha -n+2\right) \left[ (^{C}D_{^{a^{+}}}^{\alpha }f)(x)+\left(
-1\right) ^{\alpha }(^{C}D_{^{b^{-}}}^{\alpha }f)(x)\right] }{2}\ 
\end{equation*}
\end{theorem}

\begin{proof}
\bigskip First of all, since $\left( x-t\right) >0$ we can write the
following inequality

\begin{equation*}
\left( x-t\right) ^{n-\alpha -1}+\frac{1}{\left( x-t\right) ^{n-\alpha -1}}%
=\left( x-t\right) ^{n-\alpha -1}+\left( x-t\right) ^{\alpha -n+1}>2
\end{equation*}

Now \ If we multiply each side of the final inequality by $f^{\left(
n\right) }>0\;$and then integrate it over $\left[ a,b\right] $ we have
\end{proof}

\bigskip 
\begin{eqnarray*}
2\dint_{a}^{b}f^{\left( n\right) }\left( t\right) dt &<&\dint_{a}^{b}\left(
x-t\right) ^{n-\alpha -1}\;f^{\left( n\right) }\left( t\right)
dt+\dint_{a}^{b}\left( x-t\right) ^{\alpha -n+1}f^{\left( n\right) }\left(
t\right) dt \\
&=&\dint_{a}^{x}\left( x-t\right) ^{n-\alpha -1}\;f^{\left( n\right) }\left(
t\right) dt+\dint_{x}^{b}\left( x-t\right) ^{\alpha -n-1}\;f^{\left(
n\right) }\left( t\right) dt \\
&&+\dint_{a}^{x}\left( x-t\right) ^{\alpha -n+1}\;f^{\left( n\right) }\left(
t\right) dt+\dint_{x}^{b}\left( x-t\right) ^{\alpha -n+1}\;f^{\left(
n\right) }\left( t\right) dt \\
&=&\Gamma \left( n-\alpha \right) (^{C}D_{^{a^{+}}}^{\alpha }f)(x)+\left(
-1\right) ^{n}\Gamma \left( n-\alpha \right) (^{C}D_{^{b^{-}}}^{\alpha }f)(x)
\\
&&+\Gamma \left( \alpha -n+2\right) (^{C}D_{^{a^{+}}}^{\alpha }f)(x)+\left(
-1\right) ^{\alpha }\Gamma \left( \alpha -n+2\right)
(^{C}D_{^{b^{-}}}^{\alpha }f)(x)
\end{eqnarray*}

Taking into account definition $\left( \ref{d3}\right) $ we obtain
inequality (\ref{d5})

\begin{theorem}
Let $\alpha >0,$and $\alpha \notin $ $\{1,2,3,...\}$, $n=[\alpha
]+1,\;f^{\left( n\right) }>0.$\ If \ $f:I\subset \mathbb{R}\rightarrow 
\mathbb{R},\ I\subset \lbrack 0,\infty )$ ,be a differentiable function on I
such that $f\in AC^{n}L[a,b]$ $.$
\end{theorem}

where $a,b\in I$ with $0<t\leq a\leq x\leq b$. Then the following inequality
holds :

\begin{equation}
\dint_{a}^{b}\sqrt{\left\vert \left( x-t\right) \right\vert ^{2\left(
n-\alpha \right) }dt}\leq \Gamma \left( n-\alpha +1\right) \frac{\left[
(^{C}D_{^{a^{+}}}^{\alpha }f)(x)+\left( -1\right)
^{n}(^{C}D_{^{b^{-}}}^{\alpha }f)(x)\right] }{2}  \label{d0}
\end{equation}

\begin{proof}
According to relation between the Geometric and Arithmetic means we can
write the basic inequality:%
\begin{eqnarray*}
\sqrt{\left\vert \left( x-t\right) \right\vert ^{2\left( n-\alpha \right) }}
&=&\sqrt{\left\vert \left( x-t\right) \right\vert ^{\left( n-\alpha \right)
}\left\vert \left( t-x\right) \right\vert ^{\left( n-\alpha \right) }} \\
&\leq &\frac{1}{2}\left[ \left\vert \left( x-t\right) \right\vert ^{\left(
n-\alpha \right) }+\left\vert \left( t-x\right) \right\vert ^{\left(
n-\alpha \right) }\right] \\
&\leq &\frac{1}{2}\left[ \left\vert \left( x-t\right) \right\vert ^{\left(
n-\alpha \right) }+\left\vert \left( t-x\right) \right\vert ^{\left(
n-\alpha \right) }\right] f^{\left( n\right) }\left( t\right) \\
&=&\frac{1}{2}\left[ \left\vert \left( x-t\right) \right\vert ^{\left(
n-\alpha \right) }f^{\left( n\right) }\left( t\right) +\left\vert \left(
t-x\right) \right\vert ^{\left( n-\alpha \right) }f^{\left( n\right) }\left(
t\right) \right]
\end{eqnarray*}%
Now, If we integrate both sides of the first and last terms over $\left[ a,b%
\right] $ we obtain
\end{proof}

\begin{eqnarray*}
\dint_{a}^{b}\sqrt{\left\vert \left( x-t\right) \right\vert ^{2\left(
n-\alpha \right) }}dt &\leq &\frac{1}{2}\left[ \dint_{a}^{x}\left\vert
\left( x-t\right) \right\vert ^{\left( n-\alpha \right) }f^{\left( n\right)
}\left( t\right) dt+\dint_{x}^{b}\left\vert \left( t-x\right) \right\vert
^{\left( n-\alpha \right) }f^{\left( n\right) }\left( t\right) dt\right] \\
&=&\frac{1}{2}\left[ \dint_{a}^{x}\left\vert \left( x-t\right) \right\vert
^{\left( n-\alpha \right) }f^{\left( n\right) }\left( t\right)
dt+\dint_{x}^{b}\left\vert \left( x-t\right) \right\vert ^{\left( n-\alpha
\right) }f^{\left( n\right) }\left( t\right) dt\right] \\
&&+\frac{1}{2}\left[ \dint_{a}^{x}\left\vert \left( t-x\right) \right\vert
^{\left( n-\alpha \right) }f^{\left( n\right) }\left( t\right)
dt+\dint_{x}^{b}\left\vert \left( t-x\right) \right\vert ^{\left( n-\alpha
\right) }f^{\left( n\right) }\left( t\right) dt\right] \\
&=&\frac{1}{2}\left[ \dint_{a}^{x}\left( x-t\right) ^{n-\alpha }f^{\left(
n\right) }\left( t\right) dt-\dint_{b}^{x}\left( x-t\right) ^{n-\alpha
}f^{\left( n\right) }\left( t\right) dt\right] \\
&&+\frac{1}{2}\left[ \left( t-x\right) ^{n-\alpha }f^{\left( n\right)
}\left( t\right) dt-\dint_{b}^{x}\left\vert \left( t-x\right) \right\vert
^{\left( n-\alpha \right) }f^{\left( n\right) }\left( t\right) dt\right] \\
&=&\Gamma \left( n-\alpha +1\right) \frac{\left[ (^{C}D_{^{a^{+}}}^{\alpha
}f)(x)+\left( -1\right) ^{n}(^{C}D_{^{b^{-}}}^{\alpha }f)(x)\right] }{2}
\end{eqnarray*}

This completes the proof of inequality\ $\left( \ref{d0}\right) $

\begin{theorem}
Let $f:I\subset \mathbb{R}\rightarrow \mathbb{R},I\subset \lbrack 0,\infty )$%
, be a differentiable function on $I$ such that $f^{(n+1)}\in AC^{n}L[a,b].\ 
$If and $\ \left\vert f^{\left( n+1\right) }\right\vert \ $is $m-$convex on $%
\left[ x,y\right] $ \ for $t\in \lbrack 0,1]$, then for all $\alpha >0,$and $%
\alpha \notin $ $\{1,2,3,...\}$, $n=[\alpha ]+1,\ m\in \left( 0,1\right] \;$%
we have 
\begin{eqnarray}
&&\left\vert \frac{f^{\left( n\right) }\left( a\right) +f^{n}\left( b\right) 
}{2}-\frac{\Gamma \left( \alpha -n+1\right) }{2\left( b-a\right) ^{n-\alpha }%
}\left[ (^{C}D_{^{a^{+}}}^{\alpha }f)(b)+\left( -1\right)
^{n}(^{C}D_{^{b^{-}}}^{\alpha }f)(a)\right] \right\vert  \label{d4} \\
&\leq &\frac{b-a}{2}\left( \frac{1}{2\left( n-\alpha +1\right) }+\frac{1}{%
2\left( n-\alpha +2\right) }\right) \left( \left\vert f^{\left( n+1\right)
}\left( a\right) \right\vert +m\left\vert f^{\left( n+1\right) }\left( \frac{%
b}{m}\right) \right\vert \right)  \notag
\end{eqnarray}
\end{theorem}

\begin{proof}
We know from our elementary knowledge that for $\alpha \in \lbrack 0,1]$ and 
$\forall t_{1},t_{2}\in \lbrack 0,1],\ \left\vert t_{1}^{n-\alpha
}-t_{2}^{n-\alpha }\right\vert \leq \left\vert t_{1}-t_{2}\right\vert
^{n-\alpha }\;.$
\end{proof}

\ let be

\begin{equation*}
K=\frac{f^{\left( n\right) }\left( a\right) +f^{\left( n\right) }\left(
b\right) }{2}-\frac{\Gamma \left( n-\alpha +1\right) }{2\left( b-a\right)
^{n-\alpha }}\left[ (^{C}D_{^{a^{+}}}^{\alpha
}f)(b)+(^{C}D_{^{b^{-}}}^{\alpha }f)(a)\right]
\end{equation*}

In Lemma 1, using the properties of the modulus as well as the fact that $%
\left\vert f^{\left( n+1\right) }\right\vert \ $is $m-$convex on $\left[ a,b%
\right] $ , we can write the relation below.

\begin{eqnarray*}
\left\vert K\right\vert &\leq &\frac{b-a}{2}\dint_{0}^{1}\left\vert \left(
1-t\right) ^{n-\alpha }-t^{n-\alpha }\right\vert \ \left\vert f^{\left(
n+1\right) }\left( ta+\left( 1-t\right) b\right) \right\vert dt\  \\
&\leq &\frac{b-a}{2}\dint_{0}^{1}\left\vert 1-2t\right\vert ^{n-\alpha
}\left\vert f^{\left( n+1\right) }\left( ta+\left( 1-t\right) b\right)
\right\vert dt \\
\ \ &=&\frac{b-a}{2}\dint_{0}^{1}\left\vert 1-2t\right\vert ^{n-\alpha
}\left\vert f^{\left( n+1\right) }\left( ta+m\left( 1-t\right) \frac{b}{m}%
\right) \right\vert dt \\
&\leq &\frac{b-a}{2}\left\{ 
\begin{array}{c}
\left\vert f^{\left( n+1\right) }\left( a\right) \right\vert \left(
\dint_{0}^{\frac{1}{2}}t\left( 1-t\right) ^{n-\alpha }dt+\dint_{\frac{1}{2}%
}^{1}t\left( 2t-1\right) ^{n-\alpha }dt\ \right) + \\ 
+m\left\vert f^{\left( n+1\right) }\left( \frac{b}{m}\right) \right\vert
\left( \dint_{0}^{\frac{1}{2}}\left( 1-t\right) \left( 1-2t\right)
^{n-\alpha }dt+\dint_{\frac{1}{2}}^{1}\left( 1-t\right) \left( 2t-1\right)
^{n-\alpha }dt\right)%
\end{array}%
\right\}
\end{eqnarray*}

Calculate the integrals in parentheses and multiply by their coefficients$,$%
we obtain inequality $\left( \ref{d4}\right) .$

\begin{corollary}
If \ $\alpha =n\in \{1,2,3,...\}\ $and usual derivative $f^{\left( n\right)
}\left( a\right) \ $of order n exists, then Caputo fractional derivatives $%
(^{C}D_{^{a^{+}}}^{\alpha }f)(a)\ $coincides with $f^{\left( n\right)
}\left( a\right) \;$whereas $(^{C}D_{^{b^{-}}}^{\alpha }f)(b)\;$coincides $%
f^{\left( n\right) }\left( b\right) \ $to a constant multipler $\left(
-1\right) ^{n}.\ $Thus if we choose$\;$and $\alpha =0\;$In $\left( \ref{d4}%
\right) \ $with $m=1\;$we obtain 
\begin{equation*}
\left\vert \frac{f\left( a\right) +f\left( b\right) }{2}-\frac{1}{2\left(
b-a\right) }\left[ f(b)+\left( -1\right) ^{n}f(a)\right] \right\vert \leq 
\frac{b-a}{4}\left( \left\vert f^{^{\prime \prime }}\left( a\right)
\right\vert +\left\vert f^{^{\prime \prime }}\left( b\right) \right\vert
\right)
\end{equation*}
\end{corollary}

\begin{theorem}
\bigskip Let $f:\left[ a,b\right] \rightarrow \left( -\infty ,\infty \right)
\;$be a differentiable mapping on $a<b,\ $If $\alpha >0$ and $\alpha \notin $
$\{1,2,3,...\}$, $n=[\alpha ]+1,\ q>1,$ $p=\frac{q}{q-1}\;$and\ $f^{\left(
n+1\right) }\in L\left[ a,b\right] \;$and $\left\vert \ f^{\left( n+1\right)
}\right\vert ^{q}\ $is $m-$convex, $m\in \left( 0,1\right] $ \ 
\end{theorem}

Then the following inequality holds:

\begin{eqnarray}
&&\left\vert \frac{f^{\left( n\right) }\left( a\right) +f^{\left( n\right)
}\left( b\right) }{2}-\frac{\Gamma \left( n-\alpha +1\right) }{2\left(
b-a\right) ^{n-\alpha }}\left[ (^{C}D_{^{a^{+}}}^{\alpha }f)(b)+\left(
-1\right) ^{n}(^{C}D_{^{b^{-}}}^{\alpha }f)(a)\right] \right\vert
\label{d-1} \\
&\leq &\frac{b-a}{2^{1+\frac{1}{q}}}\left( \frac{1}{\left( p\left( n-\alpha
\right) +1\right) ^{\frac{1}{p}}}\right) \left( \left\vert \ f^{\left(
n+1\right) }\left( a\right) \right\vert ^{q}+m\left\vert \ f^{\left(
n+1\right) }\left( \frac{b}{m}\right) \right\vert ^{q}\right) ^{\frac{1}{q}%
}\   \notag
\end{eqnarray}

\begin{proof}
Let the left side of Lemma 1 be $K.$
\end{proof}

Since $\alpha \in \lbrack 0,1]$ and $\forall t_{1},t_{2}\in \lbrack 0,1],\
\left\vert t_{t}^{n-\alpha }-t_{2}^{n-\alpha }\right\vert \leq \left\vert
t_{1}-t_{2}\right\vert ^{n-\alpha }\ $we can write the following inequality
with properties of modulus:

\begin{equation*}
\left\vert K\right\vert \leq \frac{b-a}{2}\dint_{0}^{1}\left\vert
1-2t\right\vert ^{n-\alpha }\left\vert f^{\left( n+1\right) }\left(
ta+\left( 1-t\right) b\right) \right\vert dt
\end{equation*}%
By applying H\"{o}lder's\ inequality to the right hand side of the above
inequality with properties of modulus and after If we use $\left\vert \
f^{\left( n+1\right) }\right\vert ^{q}\ $is $m-$convex, we have%
\begin{eqnarray*}
\left\vert K\right\vert &\leq &\frac{b-a}{2}\dint_{0}^{1}\left\vert
1-2t\right\vert ^{n-\alpha }\left\vert f^{\left( n+1\right) }\left(
ta+\left( 1-t\right) b\right) \right\vert dt \\
&\leq &\frac{b-a}{2}\left( \dint_{0}^{1}\left\vert 1-2t\right\vert ^{p\left(
n-\alpha \right) }dt\right) ^{\frac{1}{p}}\left( \dint_{0}^{1}\left\vert
f^{\left( n+1\right) }\left( ta+m\left( 1-t\right) \frac{b}{m}\right)
\right\vert ^{q}\right) ^{\frac{1}{q}} \\
&\leq &\frac{b-a}{2}\left( \frac{1}{\left( p\left( n-\alpha \right)
+1\right) ^{\frac{1}{p}}}\right) ^{\frac{1}{p}}\left( \frac{\left\vert
f^{\left( n+1\right) }\left( a\right) \right\vert ^{q}+m\left\vert f^{\left(
n+1\right) }\left( \frac{b}{m}\right) \right\vert ^{q}}{2}\right) ^{\frac{1}{%
q}} \\
&=&\frac{b-a}{2^{1+\frac{1}{q}}}\left( \frac{1}{\left( p\left( n-\alpha
\right) +1\right) ^{\frac{1}{p}}}\right) ^{\frac{1}{p}}\left( \left\vert
f^{\left( n+1\right) }\left( a\right) \right\vert ^{q}+m\left\vert f^{\left(
n+1\right) }\left( \frac{b}{m}\right) \right\vert ^{q}\right) ^{\frac{1}{q}}
\end{eqnarray*}

This completes the proof of inequality$\left( \ref{d-1}\right) .$Here it can
be easily checked that 
\begin{eqnarray*}
\left( \dint_{0}^{1}\left\vert 1-2t\right\vert ^{p\left( n-\alpha \right)
}dt\right) ^{\frac{1}{p}} &=&\frac{1}{\left( p\left( n-\alpha \right)
+1\right) ^{\frac{1}{p}}}\ ,\text{\ } \\
\left\vert f^{\left( n+1\right) }\left( a\right) \right\vert ^{q}\text{\ }%
\dint_{0}^{1}tdt &=&\frac{\text{\ }\left\vert f^{\left( n+1\right) }\left(
a\right) \right\vert ^{q}}{2},\ \ \text{\ } \\
\ \ \ \ \ \ \ \ \ \ \ \ \ \ \ \ \ \ \ \ \ \ m\left\vert f^{\left( n+1\right)
}\left( \frac{b}{m}\right) \right\vert ^{q}\text{\ }\dint_{0}^{1}\left(
1-t\right) dt &=&m\frac{\text{\ }\left\vert f^{\left( n+1\right) }\left( 
\frac{b}{m}\right) \right\vert ^{q}}{2}\text{\ }
\end{eqnarray*}

\begin{corollary}
If we write corollary 1 for the Theorem $\left( \ref{d-1}\right) $ we have 
\begin{equation*}
\left\vert \frac{f\left( a\right) +f\left( b\right) }{2}-\frac{1}{2\left(
b-a\right) }\left[ f(b)+\left( -1\right) ^{n}f(a)\right] \right\vert \leq 
\frac{b-a}{2^{1+\frac{1}{q}}}\left( \frac{1}{\left( p+1\right) ^{\frac{1}{p}}%
}\right) \left( \left\vert f^{^{\prime \prime }}\left( a\right) \right\vert
^{q}+\left\vert f^{^{\prime \prime }}\left( b\right) \right\vert ^{q}\right)
^{\frac{1}{q}}
\end{equation*}%
On the other hand , let $a_{1}=\left\vert f^{^{\prime \prime }}\left(
a\right) \right\vert ^{q},\;b_{1}=\left\vert f^{^{\prime \prime }}\left(
b\right) \right\vert ^{q}\;.\;$Here $0<\frac{q-1}{q}<1,\ $for $q>1\ .\ $%
Using the fact that $\dsum\limits_{k=1}^{n}\left( a_{k}+b_{k}\right)
^{s}\leq \dsum\limits_{k=1}^{n}a_{k}^{s}+b_{k}^{s}\ ,\ for\;\left( 0\leq
s<1\right) ,\;$
\end{corollary}

$a_{1},\;a_{2,}\;a_{3},...a_{n}\geq 0,\;b_{1},\;b_{2,}\;b_{3},...b_{n}\geq
0\;$and Considering that 
\begin{equation*}
\lim_{p\rightarrow \infty }\frac{1}{\left( p+1\right) ^{\frac{1}{p}}}=1\ 
\text{and }\lim_{q\rightarrow \infty }\frac{1}{2^{1+\frac{1}{q}}}=\frac{1}{2}
\end{equation*}

we obtain

\begin{equation*}
\left\vert \frac{f\left( a\right) +f\left( b\right) }{2}-\frac{1}{2\left(
b-a\right) }\left[ f(b)+\left( -1\right) ^{n}f(a)\right] \right\vert \leq 
\frac{b-a}{2}\left( \left\vert f\ ^{\prime \prime }\left( a\right)
\right\vert +\left\vert f^{\prime \prime }\left( b\right) \right\vert \right)
\end{equation*}

\begin{remark}
Note that the right side of Corollary 1 is a better upper bound than the
right side of Corollary 2.
\end{remark}

\begin{theorem}
Let $f:I\subset \mathbb{R}\rightarrow \mathbb{R},\ I\subset \lbrack 0,\infty
)$, be a differentiable function on $I$ such that $f^{(n+1)}\in L[a,b]$ \
with $a\leq x<y\leq b,\ t\in \left[ 0,1\right] \ $.If $f^{(n+1)}\ $is $m-$%
convex on $\left[ x,y\right] .\ $The for all $\alpha >0,m\in \left( 0,1%
\right] $%
\begin{equation}
\frac{1}{y-x}f^{\left( n\right) }\left( y\right) -\frac{\left( -1\right)
^{n}\Gamma \left( n-\alpha +1\right) }{\left( y-x\right) ^{n-\alpha +1}}%
\left( C_{D_{y^{-}}^{\alpha }}f\right) \left( x\right) \leq f\left( x\right) 
\frac{n-\alpha }{n-\alpha +2}\beta \left( 2,n-\alpha \right) +mf\left( \frac{%
y}{m}\right) \frac{1}{n-\alpha +1}.  \label{d-2}
\end{equation}
\end{theorem}

\begin{proof}
From lemma 2, we have 
\begin{eqnarray*}
\frac{1}{y-x}f^{\left( n\right) }\left( y\right) -\frac{\left( -1\right)
^{n}\Gamma \left( n-\alpha +1\right) }{\left( y-x\right) ^{n-\alpha +1}}%
\left( C_{D_{y^{-}}^{\alpha }}f\right) \left( x\right)
&=&\dint_{0}^{1}\left( 1-t\right) ^{n-\alpha }f^{\left( n+1\right) }\left(
tx+m\left( 1-t\right) \frac{y}{m}\right) dt. \\
&\leq &f\left( x\right) \dint_{0}^{1}t\left( 1-t\right) ^{n-\alpha
}dt+mf\left( \frac{y}{m}\right) \dint_{0}^{1}\left( 1-t\right) ^{2\left(
n-\alpha \right) }dt \\
&=&f\left( x\right) \beta \left( 2,n-\alpha +1\right) +mf\left( \frac{y}{m}%
\right) \frac{1}{n-\alpha +1} \\
&=&f\left( x\right) \frac{n-\alpha }{n-\alpha +2}\beta \left( 2,n-\alpha
\right) +mf\left( \frac{y}{m}\right) \frac{1}{n-\alpha +1} \\
&&
\end{eqnarray*}%
which gives the required inequality $\left( \ref{d-2}\right) $. Here we used
the property of the known function $\beta $. 
\begin{equation*}
\beta \left( 2,n-\alpha +1\right) =\frac{n-\alpha }{n-\alpha +2}\beta \left(
2,n-\alpha \right) .
\end{equation*}

\begin{corollary}
If we choose $x=a,\;y=b\;$and\ $m=1\;$in $\left( \ref{d-2}\right) \;$we have
the following inequality%
\begin{equation*}
\frac{1}{b-a}f^{\left( n\right) }\left( b\right) -\frac{\left( -1\right)
^{n}\Gamma \left( n-\alpha +1\right) }{\left( b-a\right) ^{n-\alpha +1}}%
\left( C_{D_{b^{-}}^{\alpha }}f\right) \left( a\right) \leq f\left( a\right) 
\frac{n-\alpha }{n-\alpha +2}\beta \left( 2,n-\alpha \right) +f\left(
b\right) \frac{1}{n-\alpha +1}
\end{equation*}
\end{corollary}
\end{proof}

\begin{theorem}
$\alpha >0,\ $let\ $f:I\subset \mathbb{R}\rightarrow \mathbb{R},\ I\subset
\lbrack 0,\infty )$, be a differentiable function on $I$ such that $%
f^{(n+1)}\in L[a,b]$ \ with $a\leq x<y\leq b,\ t\in \left[ 0,1\right] .$If $%
\left\vert f^{(n+1)}\right\vert ^{q}\ $is $m-$convex on $\left[ x,y\right]
,\ q>1,$ $p=\frac{q}{q-1},m\in \left( 0,1\right] $
\end{theorem}

Then

$\ \ \ \ \ \ \ \ \ \ \ \ \ \ $%
\begin{equation*}
\left\vert \frac{1}{y-x}f^{\left( n\right) }\left( y\right) -\frac{\left(
-1\right) ^{n}\Gamma \left( n-\alpha +1\right) }{\left( y-x\right)
^{n-\alpha +1}}\left( C_{D_{y^{-}}^{\alpha }}f\right) \left( x\right)
\right\vert
\end{equation*}%
\begin{equation}
\leq \left( \frac{1}{\left( n-\alpha +1\right) ^{\frac{1}{p}}}\right) \left(
\left\vert f^{(n+1)}\right\vert ^{q}\left( x\right) \beta \left( 2,n-\alpha
+1\right) +m\left\vert f^{(n+1)}\right\vert ^{q}\left( \frac{y}{m}\right) 
\frac{1}{2\left( n-\alpha \right) +1}\right) ^{\frac{1}{q}}  \label{d-3}
\end{equation}

\bigskip

\begin{proof}
\bigskip Firstly, from lemma 2 and with properties of modulus and $m-$convex
of the function $\left\vert f^{(n+1)}\right\vert ^{q},\sec $ondly If we use
power mean inequality ;\bigskip 
\begin{equation*}
\left\vert \frac{1}{y-x}f^{\left( n\right) }\left( y\right) -\frac{\left(
-1\right) ^{n}\Gamma \left( n-\alpha +1\right) }{\left( y-x\right)
^{n-\alpha +1}}\left( C_{D_{y^{-}}^{\alpha }}f\right) \left( x\right)
\right\vert \leq \dint_{0}^{1}\left\vert \left( 1-t\right) ^{n-\alpha
}f^{\left( n+1\right) }\left( tx+\left( 1-t\right) y\right) \right\vert dt.
\end{equation*}%
\begin{equation*}
=\dint_{0}^{1}\left( 1-t\right) ^{n-\alpha }\left\vert f^{\left( n+1\right)
}\left( tx+m\left( 1-t\right) \frac{y}{m}\right) \right\vert dt
\end{equation*}%
\begin{equation*}
=\left( \dint_{0}^{1}\left( 1-t\right) ^{n-\alpha }dt\right) ^{\frac{1}{p}%
}\left( \dint_{0}^{1}\left( 1-t\right) ^{n--\alpha }\left\vert f^{\left(
n+1\right) }\left( tx+m\left( 1-t\right) \frac{y}{m}\right) \right\vert ^{%
\frac{1}{q}}dt\right) ^{\frac{1}{q}}
\end{equation*}
\end{proof}

\bigskip $\leq \left( \frac{1}{\left( n-\alpha +1\right) ^{\frac{1}{p}}}%
\right) .\left( \left\vert f^{(n+1)}\right\vert ^{q}\left( x\right)
\dint_{0}^{1}t\left( 1-t\right) ^{n--\alpha }dt+m\left\vert
f^{(n+1)}\right\vert ^{q}\left( \frac{y}{m}\right) \dint_{0}^{1}\left(
1-t\right) ^{2\left( n-\alpha \right) }dt\right) ^{\frac{1}{q}}$

\begin{equation*}
=\left( \frac{1}{\left( n-\alpha +1\right) ^{\frac{1}{p}}}\right) \left(
\left\vert f^{(n+1)}\right\vert ^{q}\left( x\right) \beta \left( 2,n-\alpha
+1\right) +m\left\vert f^{(n+1)}\right\vert ^{q}\left( \frac{y}{m}\right) 
\frac{1}{2\left( n-\alpha \right) +1}\right) ^{\frac{1}{q}}\ \ \ 
\end{equation*}

\bigskip $\ \ \ \ \ \ \ \ \ \ \ \ \ \ \ \ \ \ \ \ \ \ \ \ \ \ \ \ \ \ \ \ \
\ \ \ \ \ \ \ \ \ \ \ \ \ \ \ \ \ \ \ \ \ \ \ \ \ \ \ \ \ \ \ \ \ \ $

which gives the desired inequality$\left( \ref{d-3}\right) $. Here we used 
\begin{equation*}
\beta \left( 2,n-\alpha +1\right) =\dint_{0}^{1}t\left( 1-t\right)
^{n-\alpha }dt\ \ \ \text{and \ \ }\dint_{0}^{1}\left( 1-t\right) ^{2\left(
n-\alpha \right) }dt=\frac{1}{2\left( n-\alpha \right) +1}
\end{equation*}

\begin{corollary}
If we choose $x=a,\;y=b\;$and\ $m=1\;$in $\left( \ref{d-3}\right) $%
\begin{equation*}
\left\vert \frac{1}{b-a}f^{\left( n\right) }\left( b\right) -\frac{\left(
-1\right) ^{n}\Gamma \left( n-\alpha +1\right) }{\left( b-a\right)
^{n-\alpha +1}}\left( C_{D_{y^{-}}^{\alpha }}f\right) \left( a\right)
\right\vert
\end{equation*}
\end{corollary}

\begin{equation*}
\leq \left( \frac{1}{\left( n-\alpha +1\right) ^{\frac{1}{p}}}\right) \left(
\left\vert f^{(n+1)}\right\vert ^{q}\left( a\right) \beta \left( 2,n-\alpha
+1\right) +\left\vert f^{(n+1)}\right\vert ^{q}\left( b\right) \frac{1}{%
2\left( n-\alpha \right) +1}\right) ^{\frac{1}{q}}
\end{equation*}

The result in corollary $\left( 3\right) \ $is more general than the result
in corollary $\left( 2\right) .$

\section{CONCLUSION}

Where it is known that a subset of the set of real numbers has an infinite
number of upper bounds. But, the smallest upper bound of the same set is
unique. In terms of optimization theory, the aim is to capture the supremum
of the upper bounds. Inequalities involving both right-sided and left-sided
FC derivatives of non-integer order offer new estimations for integral
inequalities under convex functions. Considering\ $(\ref{d6})$ researchers
working in this field can write the above theorems once for Liouville
derivatives.

\end{document}